\documentclass[10pt]{amsart}
\usepackage{fullpage} 
\usepackage{amssymb} 
\usepackage{graphicx} 
\usepackage{amsmath,amsthm,amssymb,latexsym} 
\usepackage{enumerate} 
\usepackage{tikz}
\usepackage{accents} 
\usepackage[bookmarks,colorlinks,breaklinks]{hyperref}
\hypersetup{linkcolor=black,citecolor=black,filecolor=black,urlcolor=black}

\newtheorem{theorem}{Theorem}[section]

\theoremstyle{definition}

\newtheorem{proposition}[theorem]{Proposition}
\newtheorem{corollary}[theorem]{Corollary}
\theoremstyle{remark}

\numberwithin{equation}{section}

\begin{document}

\setcounter{page}{1}
\renewcommand{\thefootnote}{\fnsymbol{footnote}}
\title[]{Analytic and geometric representations  of the generalized $\bf n$-anacci constants }
\author{Igor Szczyrba}
\address{School of Mathematical Sciences\\
                University \!of Northern Colorado\\
                Greeley CO 80639, U.S.A.}
\email{igor.szczyrba@unco.edu}
\author{Rafa\l\ Szczyrba}
\address{Funiosoft, LLC\\
               Silverthorne CO 80498, U.S.A.}
\email{rafals@funiosoft.com}
\author{Martin Burtscher}
\address{Department of Computer Science\\
               Texas State University\\  
                San Marcos TX 78666, U.S.A.}
\email{burtscher@txstate.edu}

\begin{abstract}
We study generalizations of the sequence of the $n$-anacci constants that consist of the ratio limits generated by linear recurrences of an arbitrary order $n$ with equal positive weights $p$. We  derive the  analytic representation of  these ratio limits and  prove that, for  a fixed $p$,  the ratio limits form  a strictly increasing  sequence converging  to $p\!+\!1$.
We also construct uniform geometric representations of the sequence of the $n$-anacci constants and generalizations thereof by using dilations of compact convex sets with varying dimensions $n$. We show that, if the collections of the sets consist of  $n$-balls,  $n$-cubes, $n$-cones, $n$-pyramids, etc., then the representations of the generalized $n$-anacci constants have clear geometric interpretations. 
\end{abstract}
 
\maketitle

\section{Introduction}
\label{sec:I}
We investigate  the weighted $n$-generalized Fibonacci sequences of a specific type defined as the linear recurrences   with equal real positive weights $p$ and real initial conditions: 
\vskip-.1in 
\begin{equation} 
  F^{(n)}_k(p)\equiv p\,\big(F^{(n)}_{k-1}(p)+\cdots + F^{(n)}_{k-n}\big),\quad  n\!\leq\!k\!\in\!\mathbb{N},\quad F^{(n)}_{k\!}(p)= a_k\!\in\! \mathbb{R},\quad 0\!\le\! k \!<\! n. \label{11}
\end{equation}
 If $p\!=\!m\!\in\!\mathbb{N}$ and $a_k\!\in\!\mathbb{N}$, formula \eqref{11} creates integer sequences with the signatures $(m,\dots,m)$ that include the $n$-generalized Fibonacci numbers  with the signatures $(1,\dots,1)$ and the Horadam sequences $w_k(a_1,a_2\,;m,-m)$ with the signatures $(m,m)$, cf.~\cite{horadam} and \cite{horadam1}. See \cite{tanya} and \cite{sloan1}
for properties and contemporary  applications of the Horadam sequences with $2\!\le\!m\!\le\!10$. 

We focus on studying  the limits  of the ratios of the successive terms generated by  \eqref{11}, i.e.,  
 \vskip-.1in 
\begin{equation}
  \Phi^{(n)}(p)\equiv\lim_{k\to\infty} F^{(n)}_{k+1}(p)/F^{(n)}_{k}(p),\quad k > k_0,  \label{12}
\end{equation}

\noindent where  $k_0$ is the  biggest index for which  $F^{(n)}_{k_0}(p)\!=\!0$.
The characteristic polynomial of \eqref{11}
\vskip-.1in 
\begin{equation}
  P^{(n)}_p(\lambda)\equiv \lambda^{n}\!-b_1\lambda^{n-1}\!\cdots-b_{n}=\lambda^{n}-p\, (\lambda^{n-1}\cdots + 1) \label{13}  
\end{equation} 
has all the coefficients  $b_i\!=\!p\!\neq\!0$, so the gcd of the indices $i$ equals 1. Therefore, for any $p\!\in\!\mathbb{R}_+$ and $n\!\in\!\mathbb{N}$, polynomial \eqref{13} is asymptotically simple with the unique simple positive dominant root $\lambda^{(n)}(p)$, i.e., other roots have moduli  strictly smaller than $\lambda^{(n)}(p)$, cf.~\cite[Theorem 12.2]{Ostrowski}.  Then, as is shown in \cite {dubeau1}, limit \eqref{12}  exists for at least one initial condition  $a_{n-1}$=1, $a_k$=0, $0\!\le\!k\!<\!n\!-\!1$,\, 
and coincides with the dominant root, i.e.,   $\Phi^{(n)}(p)\!=\!\lambda^{(n)}(p)$. 

We derive the analytic representation of the set of limits $\big\{\Phi^{(n)}(p)\!=\!\lambda^{(n)}(p)\,|\,p\!\in\!\mathbb{R}_+, n\!\in\!\mathbb{N} \big\}$ 
 by proving 
\linebreak \vskip-.12in \noindent there exist a continuous function \,$\overline{\mathbb{R}}_+^2\!\ni\!(p,q)\!\to\!\lambda(p,q)\!\in\!\overline{\mathbb{R}}_+$\, such that: 

\begin{enumerate}
\item[(a)] for any $p\!\in\!\mathbb{R}_+$ and $n\!\in\!\mathbb{N}$, \,$\lambda(p,n)\!=\!\lambda^{(n)}(p)$; 

\smallskip 
\item[(b)] for any $(p,q)\!\in\!\mathbb{R} _+^2$ such that  $p\cdot q\!\neq\!1$, $\lambda(p,q)$ is of class $C^{\infty}$;

\smallskip 
\item[(c)] $\lambda(p,q)$ restricted to any line 
 with  the directional angle $0\!\leq\!\alpha\!\leq\!\pi/2$ is  strictly increasing; 

\smallskip 
\item[(d)] for any $p\!>\!0$ and $q\!\ge\!1$,\, $p\!\leq\!\lambda(p,q)\!<\!p\!+\!1$,  and for any $p\!\in\!\mathbb{R}_+,$\, $\lim_{q\to\infty}\lambda(p,q)=p+1$.\!\footnote{This  generalizes  the result regarding the limit of the $n$-anacci constants sequence: $\lim_{n\to\infty}\Phi^{(n)}(1)\!=\!2,$ ~cf. \cite{byrd},  \cite{dubeau1}, \cite{flores}, and  \cite{somer} for various proofs of the latter result.  }
\end{enumerate}
\noindent The results stated  in  (d) imply that the set $\big\{\Phi^{(n)}(m)\,|\, m,n\!\in\!\mathbb{N}\big\}$  is totally ordered  as follows: if $n_2\!>\!n_1$, then \,$\Phi^{(n_2)}(m_2)\!>\!\Phi^{(n_1)}(m_1)$\, for any $m_2$ and $m_1$, whereas $\Phi^{(n)}(m_2)\!>\!\Phi^{(n)}(m_1)$\, if $m_2\!>\!m_1$. We will refer to the  elements of this ordered set as the $(m,n)$-anacci constants. 

We also show that for any $p\!\in\!\mathbb R_+$, the set of limits $\big\{\Phi^{(n)}(p)\,|\, p\!\in\!\mathbb{R_+}, n\!\in\!\mathbb{N}\big\}$ can be represen-\linebreak ted geometrically by means of 
the dilations transforming infinite collections of compact convex sets with increasing dimensions $n$ about homothetic centers contained in the sets but not being their centers of mass.  
Such  representations have clear geometric interpretations if 
the centers of mass of the sets  are determined  by a simple formula in terms of some boundary points.

For example, in the  $n$-balls, $n$-cubes, (finite)  $n$-cones,   $n$-pyramids, and generally in the compact convex 
$n$-polytopes, cf.~\cite{coxter},  the centers of mass  divide  the interval linking some boundary points according to  the ratio $1\!\!:\!\!1$ or $n\!:\!1$. 
We construct two  geometric  representations of the $(m,n)$-anacci constants $\Phi^{(n)}(m)$ using the $n$-balls and $n$-cones. Both of these  representations have clear geometric interpretations correlated  with the order introduced  above. 
 
The   geometric representations of the $\Phi^{(n)}(m)$'s can be extended  to the  representations of the limits $\Phi^{(n)}(p)$ by substituting $p\!\in\!\mathbb R_+$ for $m\!\in\!\mathbb N$. A different geometric representation of the $n$-anacci constants $\Phi^{(n)}(1)$ by means of the  $n$-parallelepipeds  has been  introduced  in \cite{dubeau}.

\vskip 0.25in
\section{ Analytic representation of the  ({\it m,n})-anacci constants } 
\label{sec:A}
The limits  $\Phi^{(n)}(p)\!=\!\lambda^{(n)}(p)$  are also roots of the  polynomials  
\begin{equation}
Q^{(n)}_{\,p}(\lambda)\equiv \lambda^{n+1} - (p+1)\lambda^n + p=(\lambda-1)P^{(n)}_p(\lambda). \label{21}
\end{equation}
\noindent We derive the analytic representation of the set $\big\{\Phi^{(n)}(p)\,|\, p\!\in\!\mathbb{R}_+,\, n\!\in\!\mathbb{N}\big\}$ using  the  function 
\begin{equation}
 Q(\lambda,p,q)\equiv \lambda^{q+1} - (p+1)\lambda^q + p,\quad  \lambda, p, q\in\mathbb{R}_+. \label{22} 
\end{equation} 
\noindent The function $Q(\lambda,p,q)$ equals 0 at the plane $\lambda\!=\!1$ and at  the roots  $\lambda^{(n)}(p)$, i.e., in particular,  the restriction $Q(\lambda,1,q)$ of $Q(\lambda,p,q)$ to  the plane $p\!=\!1$ includes the $n$-anacci constants $\Phi^{(n)}(1)$.
 
Fig.~1 depicts, in the sub-domain  where $0\!<\!\lambda\!\leq\!2$ and  $0\!<q\!\leq\!4$, the restriction  $Q(\lambda,1,q)$ 
 and the function \,$O(\lambda,q)\!\equiv\!0$.  The functions   intersect along  the zero line $O(1,q)$ and the zero curve, say  $\lambda_1(q)\!=\!0$, which is defined implicitly by the equation $Q(\lambda,1,q)\!=\!0$.

\vskip.1in
\begin{figure}[h] 
  \centering=
  \includegraphics[width=5.67in,height=2.36in,keepaspectratio]{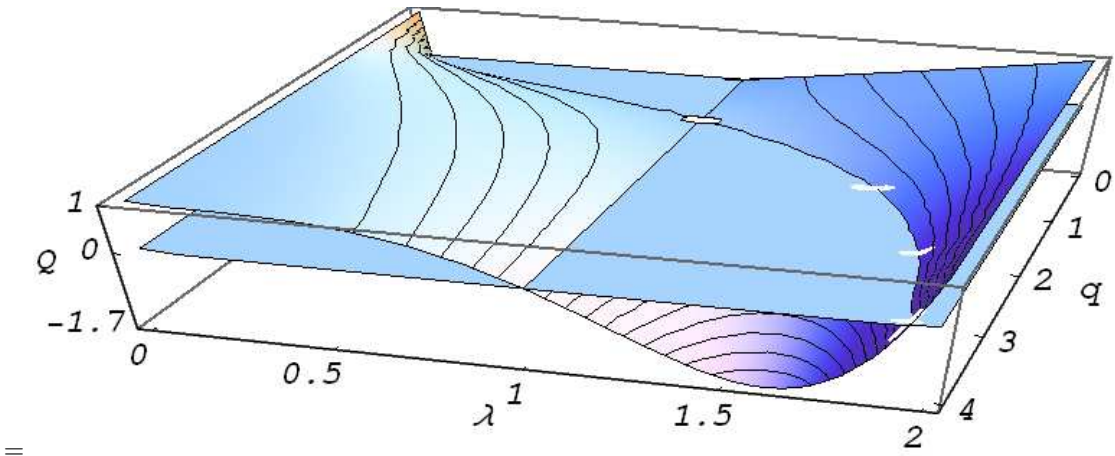}
  \caption{The  restriction $Q(\lambda,1,q)$  of the function $Q(\lambda,p,q)$ and the function $O(\lambda,q)\!\equiv\!0$  intersecting along the line $O(1,q)$  and the zero curve  $\lambda_1(q)$. The oval white  mark at the crossing of $O(1,q)$ and $\lambda_1(q)$ indicates the location of the golden ratio \,$\Phi\!=\!\Phi^{(2)}(1)$. The other white marks indicate the location of the  $n$-anacci constants \, $\Phi^{(n)}(1)$\, with   $n\!=\!2,3,4$.} 
\end{figure}

 \noindent  The equation  $Q(\lambda,a,q)\!=\!0$, $a\!\in\!\mathbb R_+$, defines the zero curve $\lambda_a(q)$ that  contains the sequence of $(m,n)$-anacci constants  $\Phi^{(n)}(m)_{n=1}^\infty$ if $a\!=\!m$.   
 Next two propositions establish the analytic representation of all zeros of  the function $Q(\lambda,p,q)$. 

\begin{proposition} 
For any given $p,q\!\in\!\mathbb{R}_+$, $p\cdot q\!\neq\! 1$, the function $Q(\lambda ,p,q)$ of one variable $\lambda$ has the unique zero $\lambda(p,q)\!\neq\!1$, whereas if\, $p\cdot q\!=\! 1$, its unique zero $\lambda(p,1/p)\!=\!1$. Moreover,
\begin{equation}
1<(p+1)q/\!(q+1)<\lambda(p,q)  \quad \text{iff}\quad  p\cdot q>1; \label{23}
\end{equation}
\begin{equation}
0<\lambda(p,q)<(p+1)q/\!(q+1) < 1  \quad \text{iff}\quad  p\cdot q<1; \label{24}
\end{equation}
\begin{equation}
\lambda(p,q) <p+ 1 \quad \text{for any} \quad q\in\mathbb R_+. \label{25}
\end{equation}
\end{proposition} 

\begin{proof} 
The partial derivative of function \eqref{22} with respect to $\lambda$ is given by
\begin{equation}
 \partial Q(\lambda,p,q)\!/\!\partial\lambda = \lambda^{q-1}\big(\lambda(q+1)-(p+1)q\big). \label{26}
\end{equation}
Thus, for any $p,q\!\in\!\mathbb{R}_+$, the function $Q(\lambda,p,q)$ of the variable $\lambda\!>\!0$  has one local minimum at
\begin{equation}
\lambda_{min}(p,q)=(p+1)q/\!(q+1). \label{27}
\end{equation} 
\noindent 
Formula  \eqref{27} implies that the minimum  is assumed at $\lambda_{min}\!=\!1$\,   iff\, $p\cdot q\!=\!1$. In this case, 1 is\linebreak the only zero of the function $Q(\lambda,p,q)$  of the variable $\lambda$ because  function \eqref{22}  equals zero at $\lambda\!=\!1$,  cf.~the most left white  oval mark \,$\Phi^{(1)}(1)\!\!=\!1$ \,in Fig.~1. 
If\, $p\cdot q\!\neq\!1$, then there must exist a second {\em positive\/} zero\, $\lambda(p,q)$  of \,$Q(\lambda,p,q)$\, besides 1 (if $\lambda_{min}\!<\!1$, the  existence of the positive zero is implied by the fact that $Q(0,p,q)\!=p\!>\!0$), cf.,  Fig.~1.  Moreover,  the following  holds
\begin{equation}
 1\!<\!\lambda_{min}(p,q)\!<\!\lambda(p,q)\quad\text{iff}\quad p\cdot q\!>\!1\quad\text{and then} \quad
 Q(\lambda,p,q)\!<\!0\quad \text{iff}\quad 1\!<\!\lambda<\!\lambda(p,q);\label{28}
\end{equation}
\begin{equation}
 0\!<\!\lambda(p,q)\!<\!\lambda_{min}(p,q)\!< \!1\quad \text{iff}\quad p\cdot q\!<\!1 \quad \text{and then}\quad Q(\lambda,p,q)\!<\!0\quad\text{iff}\quad \lambda(p,q)\!<\!\lambda\!<\!1;\label{29}\
\end{equation}
\begin{equation}
\lambda(p,q)\!=\!\lambda_{min}(p,q)\!=\!1\quad \text{iff}\quad p\cdot q\!=\!1 \quad \text{and then} 
\quad Q(\lambda,p,q)\!>\!0\quad\text{iff
}\quad \lambda\!\neq\! 1. \label{210}
\end{equation}
\vskip.05in\noindent Since  $Q(p\!+\!1,p,q)\!=\!p\!>\!0$, formulas \eqref{28}--\eqref{210} imply  that $\lambda(p,q)\!<\!p\!+\!1$ for any $q\!\in \mathbb R_+$. 

\end{proof}

\smallskip
\begin{proposition}
{\rm(i)} The assignment $\mathbb{R}_+^2\!\ni\!(p,q)\to\lambda(p,q)\!\in\!\mathbb{R}_+$ defines a continuous function

\hskip .33in such that, for any $(p,n)\!\in\!\mathbb{R}_+\!\times\!\mathbb{N}$,  $\lambda(p,n)\!=\!\lambda^{(n)}(p)$ holds;
\begin{enumerate}

\vskip-.1in
\item[(ii)] \,if $p\cdot q\!\neq\!1$, the function $\lambda(p,q)$ is of class $C^{\infty}$\! and its restriction $\lambda(p,q)|_\ell$ to  any line \,$\ell$ 
in the domain with  the directional angle $0\!\leq\!\alpha\!\leq\!\pi/2$ is strictly increasing; 

\smallskip
\item[(iii)] for any $p\!>\!0$ and $q\!\ge\!1$, $p\!\leq\!\lambda(p,q)$,  and for any $p\!\in\!\mathbb{R}_+,$ $\lim_{q\to\infty}\lambda(p,q)=p+1$;

\smallskip
\item[(iv)] for any $p_0\!\in\!\mathbb{R}_+$,  $\lim_{(p,q)\to(p_0,0)}\!\lambda(p,q)\!=\!0$, and for any $ q_0\!\in\!\mathbb{R}_+$, $\lim_{(p,q)\to(0,q_0)}\!\lambda(p,q)\!=\!0$,
 i.e., the open domain ${\mathbb{R}}_+^2$ of\, $\lambda(p,q)$ can be extended to the closed domain $\overline{\mathbb{R}}_+^2$.
\end{enumerate}
\end{proposition}

\begin{proof} 
(i) \,It follows from Proposition 2.1  that the  assignment  defines a function.   If, for $(p_0,q_0)$ with $p_0\cdot q_0\!=\!1$, $\lim_{(p,q)\to(p_0,q_0)}\!\lambda(p,q)\!\neq\!1\!=\!\lambda(p_0,q_0)$, then we have a contradiction with  \eqref{210} due to the continuity of the function $Q(\lambda,p,q)$ and the fact that  $Q\big(\lambda(p,q),p,q\big)\!=\!0$. Thus, $\lambda(p,q)$ is continuous at $(p_0,q_0)$ with $p_0\cdot q_0\!=\!1$. The continuity of $\lambda(p,q)$ at $(p_0,q_0)$ with $p_0\cdot q_0\!\neq\!1$ is implied by part (ii). The definition of the function $\lambda(p,q)$  assures that  $\lambda(p,n)\!=\!\lambda^{(n)}(p)$.

\smallskip
(ii) The equation $Q\big(\lambda(p,q),p,q\big)\!=\!0$ defines the function $\lambda(p,q)$ implicitly. 
It follows from  formulas \eqref{26} and \eqref{27} that the partial derivative $\partial Q(\lambda,p,q)/\partial\lambda$ is continuous and equals 0 iff\,  $\lambda(p,q)\!=\!\lambda_{min}(p,q)$, i.e.,  according to \eqref{210} iff\, $p \cdot q\!=\!1$. Thus, the implicit function theorem implies that if\, $p\cdot q\!\neq\!1$, the function  $\lambda(p,q)$ is continuously differentiable and 
\begin{equation}
\frac{\partial\lambda(p,q)}{\partial p}=\frac{-1}{\frac{\partial Q\big(\lambda(p,q),p,q\big)}{\partial\lambda}}\cdot \frac{\partial Q\big(\lambda(p,q),p,q\big)}{\partial q}=\frac{1-\big(\lambda(p,q)\big)^q}{\left[\lambda(p,q)(q+1)-(p+1)q\right]\big(\lambda(p,q)\big)^{q-1}}, \label{211}
\end{equation} 
\begin{equation}
\frac{\partial\lambda(p,q)}{\partial q}=\frac{-1}{\frac{\partial Q\big(\lambda(p,q),p,q\big)}{\partial\lambda}}\cdot \frac{\partial Q\big(\lambda(p,q),p,q\big)}{\partial p}=\frac{\left[p+1-\lambda(p,q)\right]\big(\lambda(p,q)\big)^{q}\ln q} {\left[\lambda(p,q)(q+1)-(p+1)q\right]\big(\lambda(p,q)\big)^{q-1}}. \label{212}
\end{equation} 
\noindent Since the function $\lambda(p,q)$ is continuously differentiable if $p\cdot q\!\neq\!1$, it follows from formulas \eqref{211} and  \eqref{212} that all partial derivatives of $\lambda(p,q)$ of an arbitrary  order 
exist and are continuous. Consequently, the function   $\lambda(p,q)$ is of class $C^{\infty}$ if\, $p\cdot q\!\neq\!1$.   

If $p\cdot q\!>\!1$ (respectively $p\cdot q\!<\!1$), then the denominator and, according to  \eqref{23}--\eqref{25}, both numerators in  \eqref{211} and \eqref{212} are positive (respectively negative). Thus, the directional derivative of  $\lambda(p,q)$ along $\ell$ with the listed property is positive, i.e, $\lambda(p,q)|_\ell$ is strictly increasing if $p\cdot q\!\neq\!1$.  It follows from Proposition 2.1  that  $\lambda(p,q)|_\ell$ is also strictly increasing  at  $p\cdot q\!=\!1$. 

\smallskip(iii) Formula \eqref{13} implies that  \,$\lambda(p,1)\!=\!p$, which is smaller than $\lambda(p,q)$, $q\!>\!1$, since for a fixed  $p$,  $\lambda(p,q)$ is strictly increasing.   The convergence to $p\!+\!1$ follows from   \eqref{23} and \eqref{25}.

\smallskip (iv) The first limit equals 0 due to formulas  \eqref{27} and \eqref{29}. Definition \eqref{22} implies that the second limit is equal either to 0 or 1. The latter is impossible due to \eqref{27} and \eqref{29}.  

\end{proof}

The lower bounds $(p\!+1\!)q/\!(q\!+\!1)$ for $\lambda(p,q)$ in \eqref{23}  predict, e.g.,  that  the golden ratio  $\Phi\!\equiv\!\lambda(1,2)$ is only greater than 4/3.
The next  proposition provides  more subtle lower bounds for $\lambda(p,q)$, which imply that  the values  $\lambda(p,q)$ are close to $p\!+\!1$ already  when $q$ is small.

\smallskip
\begin{proposition} 
  For any $q\!\ge\!2$  and   $p\!>\!1\!/\!\Phi$, it holds
\begin{equation}
p+\!1\!-\!1\!/\!(p+\!1)\!<\!\lambda(p,q) \label{213}  
\end{equation}
  moreover, the lower bounds \eqref{23} and \eqref{213} satisfy
\begin{equation}
 (p+1\!)q/\!(q+\!1)\leq p+\!1\!-\!1\!/\!(p+\!1) \quad\text{iff}\quad q\!\leq\!(p+1)^2-1. \label{214}
\end{equation} 
\end{proposition}

\begin{proof}
 Since\, $Q^{(2)}_p\big(p\!+\!1\!-\!1\!/\!(p\!+\!1)\big)\!=\!-p(p^2\!+\!p\!-1)\!/\!(p\!+\!1)^2\!<\!0$\,  if \,$p\!>\!1\!/\!\Phi$, formula  \eqref{28} implies  that  $p+\!1\!-\!1\!/\!(p+\!1)\!<\!\lambda(p,2)$. Now, for $q\!>\!2$, $\lambda(p,2)\!<\!\lambda(p,q)$ because $\lambda(p,q)$ is strictly increasing for a fixed $p$.
 Formula \eqref{214} follows from a simple calculation.

\end{proof}

\smallskip
Fig.~2 depicts the restrictions  $\lambda(a,q)$
 of $\lambda(p,q)$ to the lines $p\!=\!a$, where  $a\!=\!\frac{2}{3},1,\dots,2\frac{2}{3}$.\footnote{The restrictions $\lambda(a,q)$ are the same as  the zero curves $\lambda_a(q)$ introduced  above.}   Each restriction  $\lambda(a,q)$ starts at $\lambda(a,1)\!=\!a$ and increases asymptotically to $a\!+\!1$. It is within the distance $1\!/\!(a\!+\!1)$ from   $a\!+\!1$ in the region to the right of the  line  $q\!=\!2$ (not marked). This lower bound exceeds  the lower bound given by formula \eqref{23}
 above the white curve  $q\!=\!(p\!+\!1)^2\!-\!1$ that  goes through the points  ($3,1)$ and  $(4,2\!/\!\Phi)$. The restrictions $\lambda(a,q)$ with an integer $a$ include the  $(m,n)$-anacci constants  $\Phi^{(n)}(m)$. 

\begin{figure}[h] 
  \centering
  \includegraphics[width=5.67in,height=2.4in,keepaspectratio]{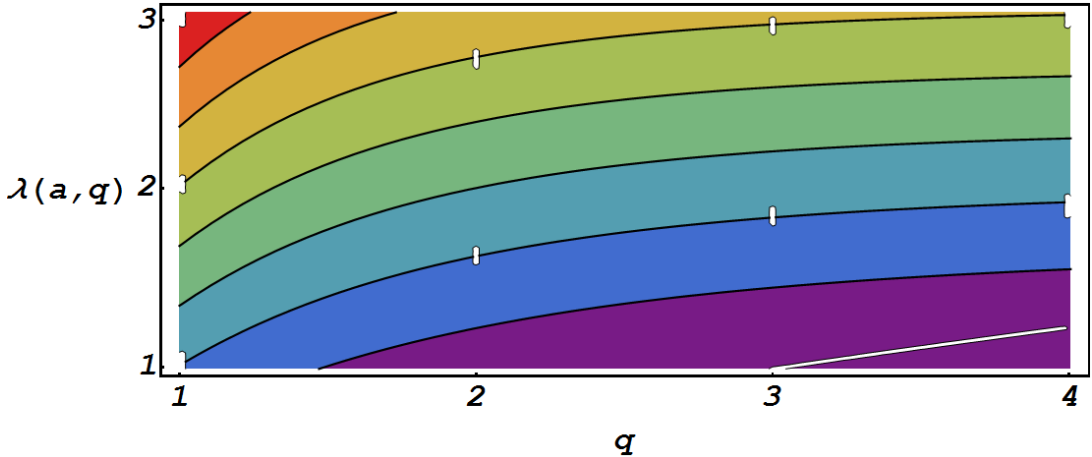}
  \caption{The restrictions $\lambda(a,q)$,   $a\!=\!\frac{2}{3},1,\dots,2\frac{2}{3}$, $1\!\leq\!q\!\leq\!4$. The  white  marks indicated the $(m,n)$-anacci constants  $\Phi^{(n)}(m)$, $m\!=1,2$, $n\!=\!1,\dots,4$.  The white curve $q\!=\!(p\!+\!1)^2\!-\!1$} determines the region where the lower bound \eqref{213} exceeds  the lower bound \eqref{23}.
\end{figure}

The function $\lambda(p,q)$ is also defined implicitly by the continuous function 
\begin{equation}
p(\lambda,q)\!\equiv\! \lambda^q(\lambda-1)\!/\!(\lambda^q-1) \quad \text{if}\quad   \lambda\!\ne\!1 \quad \text{and} \label{215}
\end{equation}
\begin{equation}
  p(\lambda,q)\!\equiv\!1\!/\!q \quad\text{if}\quad  \lambda\!=\!1,\label{216} 
\end{equation}
\noindent which  is  of class \,$C^{\infty}$ if $\lambda\!\neq\!1$.

\smallskip
For $q\!=\!n\!\in\!\mathbb N$, the function defined by  \eqref{215}--\eqref{216}  takes the form  
\begin{equation}
p(\lambda,n)=\frac{\lambda^n}{\Sigma_{k=0}^{n-1}\,\lambda^k}, \quad n\!\in\!\mathbb{N}.\label{217}  
\end{equation}
 
\smallskip 
Propositions 2.1--2.3 and formula \eqref{217} imply the following 
\begin{corollary} 
{\rm(i)} The function  $\lambda(p,q)$  is  concave down  where $p\cdot q\!\ge\!1$, i.e., in the sub-domain
   
\hskip.33in of \,$\overline{\mathbb{R}}_+^2$ limited by the hyperbola $p\!\cdot\! q\!=\!1$; 

\begin{enumerate}
\item[(ii)] the  plane $\overline{\mathbb{R}}_+^2\!\ni(p,q)\!\to\!\mathcal P(p,q)\!\equiv\!p\!+\!1$  majorizes the function $\lambda(p,q)$  from above and is the   asymptotic  plane for $\lambda(p,q)$;  

\smallskip 
\item[(iii)] the triple $\big(\lambda(p,q),p,q\big)\!\in\!\mathbb{N}^3$ iff it equals $(m,m,\!1)$  with $m\!\in\!\mathbb{N}$, i.e.,  the $(m,n)$-anacci constants $\Phi^{(n)}(m)$ are integer iff $n\!=\!1$; 

\smallskip 
\item[(iv)]  if the function $\lambda(p,n)\!=\!m\!\in\!\mathbb{N}$  for some $n\!\in\!\mathbb N$,  then  $p$ is rational and 
 $(m\!-\!1)\!<\!p\!<\!m$;

\smallskip 
\item[(v)] the  sequences \,$\big(\Phi^{(n)}(m)\big)_{n=1}^{\infty}$ with a fixed   $m\!\in\!\mathbb{N}$, $\big(\Phi^{(n)}(m)\big)_{m=1}^{\infty}$ with a fixed  $n\!\in\!\mathbb{N}$, as well as $\big(\Phi^{(n)}(kn)\big)_{n=1}^{\infty}$ and $\big(\Phi^{(km)}(m)\big)_{m=1}^{\infty}$ with a fixed $k\!\in\!\mathbb{N}$ are strictly increasing;

\smallskip
\item[(vi)]  for any $n\!\in\!\mathbb{N}$, the sequence $\big(\frac{m\!+\!1}{m}\Phi^{(n)}(m)\big)_{m=1}^{\infty}$  is strictly increasing,  cf.~Appendix A;
\smallskip
\item[(vii)] if $n\!>\!1$,
the sequence $\big(\frac{1}{m}\Phi^{(n)}(m)\big)_{m=1}^{\infty}$ is strictly decreasing to $1$, cf.~Appendix B.
\end{enumerate}
\end{corollary}
Fig.~3 depicts the asymptotic  plane $\mathcal P(p,q)\!\equiv\!p+\!1$ and the function $\lambda(p,q)$ generated by using  formulas \eqref{215} and \eqref{216} in the sub-domain where $0\!\leq\!p\!\leq\!3$ and  $0\!\leq\!q\!\leq\!4.1$. The thick curves and the curve at the rim of the graph are the restrictions $\lambda(a,q)$ of $\lambda(p,q)$ with  $a\!=\!\frac{1}{3},\frac{2}{3}\dots,2\frac{2}{3},3$. They include the $(m,n)$-anacci constants $\Phi^{(n)}(m)$, $m\!=\!1,2,3$,  $n\!=\!1,2,3,4$.  

The thin curves increasing from left to right are the restrictions  $\lambda(p,q)|_\ell$ of $\lambda(p,q)$ to four lines $\ell$ in the domain. These restrictions are not straight lines. Only the restriction to the  line $q\!=\!1$,  $\lambda(p,1)\!=\!p$, is   a straight line. 
The horizontal thin curves are the level curves  $\lambda(p,q)\!=\!c,$ $c\!=\!\frac{1}{2},1\dots,3\frac{1}{2}, 4$. Among them, only the level curve $\lambda(p,q)\!=\!1$ is a hyperbola.  

\begin{figure}[h]
  \centering
  \includegraphics[width=5.67in,height=1.99in,keepaspectratio]{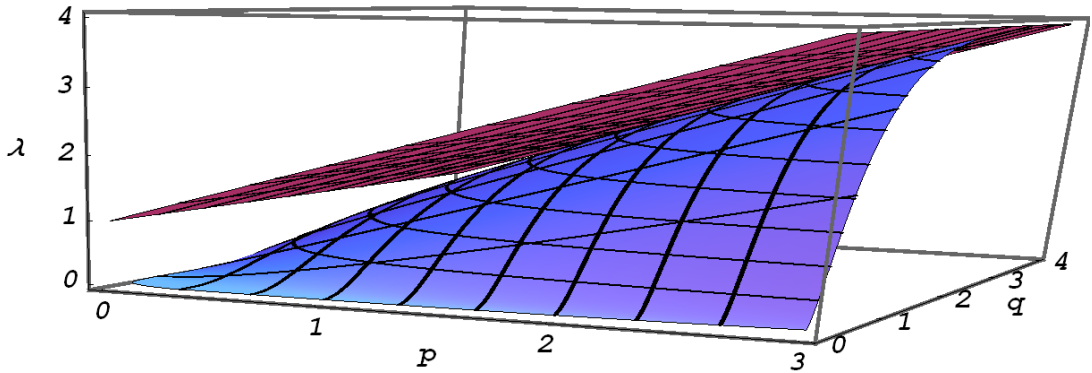}
  \caption{The function $\lambda(p,q)$ providing the analytic realization of the $(m,n) $-anacci constants  and the asymptotic  plane $\mathcal P(p,q)$ in the sub-domain where   $0\!\leq\!p\!\leq\!3$ and  $0\!\leq\!q\!\leq\!4.1$.  The  $\Phi^{(n)}(m)$'s with  $m\!=\!1,2,3$ and $n\!=\!1,2,3,4$ lie in the thick  curves and the curve at the rim of the graph. The thin curves increasing from left to right are the restrictions  $\lambda(p,q)|_\ell$ of $\lambda(p,q)$. The horizontal thin curves are the level curves  $\lambda(p,q)\!=\!c$ with $c\!=\!\frac{1}{2},1\dots,3\frac{1}{2}, 4$.}
\end{figure}

\newpage

\section{Geometric representations of the ({\it m,n})-anacci constants } 
\label{sec:G}
Let $\Lambda\!\in\!{L}(\mathbb{R}^n\!,\mathbb{R}^n)$ denote  a dilation  with a dilation factor $\lambda\!>\!0$ acting in the Euclidean space $\mathbb{R}^n$ and let  $O$ be the  homothetic center of the dilation $\Lambda$. Let $\mathfrak A\!\subset\!\mathbb{R}^n$ be  a compact convex 
$n$-dimensional set with the  center of mass $A$.
If $O\!\in\!\mathfrak A$, then  $\mathfrak A$ is contained in  the image $\Lambda(\mathfrak A)$ as a proper subset iff $\lambda\!>\!1$, whereas $\Lambda(\mathfrak A)\!\subsetneq\mathfrak A$ iff  $0\!<\!\lambda\!<\!1$. Thus, if $\lambda\!>\!1$, we define  the non-empty  set $\mathfrak B\!\equiv\!\Lambda(\mathfrak A)\!\setminus\!\mathfrak A$, and if $0\!<\!\lambda\!<\!1$, we define it as  $\mathfrak B\!\equiv\mathfrak A \mathfrak\!\setminus\!\Lambda(\mathfrak A)$. 

We construct geometric representations of the limits  $\Phi^{(n)}(p)\!=\!\lambda^n(p)$  using  dilations of an infinite collection of  compact convex sets $\mathfrak A(p,n)\!\subset\!\mathbb{R}^n\!$, $n\!=1,2,\dots$, and analyzing, for any $p\!\in\!\mathbb{R_+}$ and $n\!\in\!\mathbb{N}$, relations among the distances  $d(\cdot,\cdot)$ between the following four points in $\mathbb{R}^n$:

\smallskip 
\begin{enumerate}  \item[(a)] the center of mass $A$ of the set $\mathfrak A$ in the collection corresponding to given $p$ and $n$, 

\smallskip \item[(b)] the homothetic center  $O\!\in\!\mathfrak A$, $O\!\not =\!A$, of a dilation $\Lambda$ with $\lambda\!>\!0$, 

\smallskip \item[(c)] the center of mass ${\Lambda(A)}$ of the image $\Lambda(\mathfrak A)$, and 

\smallskip\item [(d)] the center of mass $B(\lambda)$ of the set $\mathfrak B$.
\end{enumerate}

\smallskip 
The centers $O,\, A,\, \Lambda(A)$, and
$B(\lambda)$ lie on the line $\mathcal{L}(O,A)$  because the dilations are linear transformations.  The set $\mathfrak B$ is constructed by removing some mass from the set $ \Lambda(\mathfrak A)$ if $\lambda\!>\!1$ (respectively from  $\mathfrak  A$ if $\lambda\!<\!1$). \linebreak Thus,  the distances satisfy   $d(O,A)\!<\!d\big(O,\Lambda(A)\big)\!<\!d(O,B(\lambda))$  if $\lambda\!>\!1$, cf.~Fig.~4,  whereas if $\lambda\!<\!1$, $d\big(O,\Lambda(A)\big)\!<\!d(O,A)\!<\!d(O,B(\lambda))$ holds.

For any  $\mathfrak A$ and $O$, the function $\big\{\mathbb{R}_+\!\setminus\!\{1\}\big\}\ni\!\lambda\!\to\!B(\lambda)\!\in\!\mathbb{R}^n$ is continuous and has its limit at 1 \linebreak since 
the function $\mathbb{R}_+\!\ni\!\lambda\!\to\!\Lambda(\lambda)\!\in\!{L}(\mathbb{R}^n\!,\mathbb{R}^n)$ is continuous. 
 Thus,  we can define the center of mass $B(1)\!\equiv\!\lim_{1\neq\lambda\to 1}B(\lambda)$ despite that the set $\mathfrak B$ is not defined if $\lambda\!=\!1$. For  a fixed  $O$, the center  of mass $B(1)$ lies between the centers of mass $A$ and   $B(\lambda)$ with $\lambda\!\neq\!1$, cf.~Fig.~4. 
 
\begin{figure}[h] 
  \centering
  \includegraphics[width=5.67in,height=1.89in,keepaspectratio]{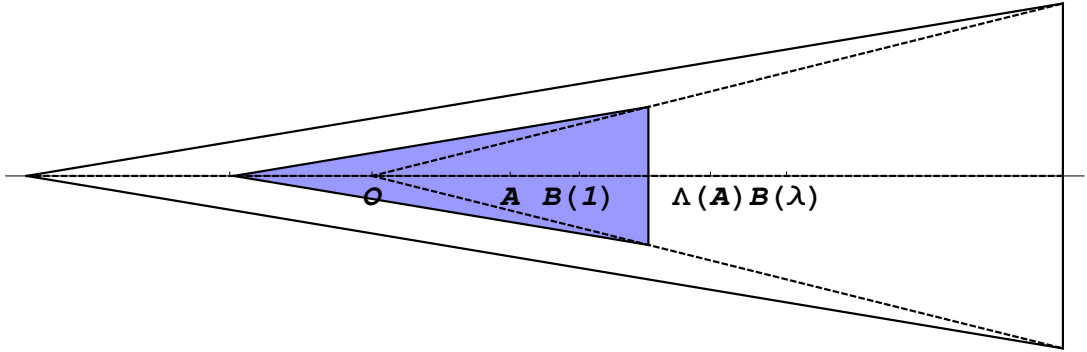}
\caption{Dilation  with $\lambda\!=\!2.5$ of\, $\mathfrak A\!\subset\!\mathbb{R}^2$  about $O\!\in\!\mathfrak A$ resulting in a hollow  set  $\mathfrak B$.}
\end{figure}
\vskip-.1in  
The next proposition provides the basis for constructing uniform geometric  representations of the limits  $\Phi^{(n)}(p)\!=\!\lambda^n(p)$ in terms of the dilation factors $\lambda$. It also determines the position of the center of mass $B(1)$  relative to  the center of mass $\Lambda(A)$.
In what follows, unless it leads to a confusion,  we denote the center of mass  $B(\lambda)$ corresponding to a given $\lambda$ just  as $B$. 
\begin{proposition}
Let\, $\mathfrak A\!\subset\!\mathbb{R}^n$\! be a compact convex set with the center of mass $A$. Let $\Lambda$ be a dilation  about a homothetic center $O\!\in\!\mathfrak A$ such that  $O\!\not =\!A$, and let $B$ be the center of mass of the set $\mathfrak B$. Then the following holds

\smallskip
\begin{enumerate} \item[(i)] if  $\,\lambda\!>\!1$,\, then\,\, $\lambda\!=\!\lambda^{(n)}(p)$\, 
iff \, \,$d(A,B)/d(O,A)\!=\!p$ \,where\,\, $p\!>\frac{1}{n}$, 

\item[] i.e., $\lambda^{(n)}(p)\!= \!d\big(O,\Lambda(A)\big)/d(O,A);$   

\smallskip
\item[(ii)] if\, $0\!<\!\lambda\!<\!1$,\, then \,\,$\lambda\!=\!1/\lambda^{(n)}(p)$ \,
iff \,\,$d\big(\Lambda(A),B\big)/d\big(O,\Lambda(A)\big)\!=\!p$\, where \,$p\!>\frac{1}{n}$, 

\item[] i.e.,  $\lambda^{(n)}(p)\!=\!d(O,A)/d\big(O,\Lambda(A)\big);$

\smallskip
\item[(iii)] if\, $\lambda\!=\!1$, \,then $\lambda\!=\!\lambda^{(n)}(\frac{1}{n})$ \,
iff \,  \,$d\big(A,B(1)\big)\!=\!d\big(\Lambda(A),B(1)\big)\!=\!\frac{1}{n}d\big(O,A)$. 
\end{enumerate}
\end{proposition}
\begin{proof}
(i)   If $\lambda\!>\!1$,  the centers of mass are ordered on the line $\mathcal{L}(O,A)$  by their distances from   $O$ in the following way: $O\!<\!A\!<\!\Lambda(A)\!<\!B$. Since  some mass is removed from   $\Lambda (\mathfrak A)$ to construct the set $\mathfrak B$,  the location of the center of mass $B$ coincides with the fulcrum of  the lever  that is in equilibrium  
when two forces, {\bf F}$_A$ and {\bf F}$_{\Lambda(A)}$, with magnitudes proportional to the $n$-volumes of the sets $\mathfrak A$ and $\Lambda(\mathfrak A)$, act in opposite directions at $A$ and ${\Lambda(A)}$, respectively.
 
Dilations with $\lambda\!>\!0$ change distances proportionally to $\lambda$ and  $n$-volumes  as $\lambda^n$, i.e., 

\begin{equation}
  d(O,{\Lambda(A)} )= \lambda d(O,A)\quad {\rm and} \quad \label{31}
\end{equation}
\begin{equation}
  {\bf F}_{\Lambda(A)} =  -\lambda^n{\bf F}_A. \label{32}
\end{equation}
\noindent Using  \eqref{31}, we obtain 
\begin{equation}
   d\big({\Lambda(A)},B\big)= d(O,A)+d(A,B) -  \lambda\, d(O,A),\label{33}
\end{equation}
whereas the following  equilibrium relation for the lever is implied by \eqref{32} 
\begin{equation}\label{34}
   d\big({\Lambda(A)},B\big)\lambda^n=d(A,B).
\end{equation}
 It follows from relations  \eqref{33} and \eqref{34}  that 
\begin{equation}\label{35}
    d(A,B)=\lambda^n\big[d(A,B)+(1- \lambda)d(O,A)\big]. 
\end{equation}
In turn, equation  \eqref{35} implies that  $\lambda$ is a root  of the polynomial  
\begin{equation}
 \lambda^{n+1} - (p+1)\lambda^n + p=0, \label{36}
\end{equation} 
iff \,$d(A,B)/d(O,A)\! =\!p$. Since polynomials \eqref{21} and \eqref{36} are the same and \eqref{31} holds,    $\lambda\!=\!\lambda^{(n)}(p)\!= \!d\big(O,\Lambda(A)\big)/d(O,A)$  iff $d(A,B)/d(O,A)\!=\!p$. Formula \eqref{23} implies that $p\!>\!1/n$.

\smallskip
(ii)  If $0\!<\!\lambda\!<\!1$, the centers of mass are ordered by their distances from $O$ as follows: $O\!<\!\Lambda(A)\!<\!A\!<\!B.$ Let us rename the center of mass $\Lambda(A)$ as $A'$. Then, $A\!=\!\Lambda^{-1}(A')$ where $\Lambda^{-1}$ is the dilation about the homothetic center $O$ with the dilation factor $1/\lambda\!>\!1$.  

Part (i) of the proposition implies   that $1/\lambda\!=\!\lambda^{(n)}(p)=\! d(O,A)/d\big(O,\Lambda(A)\big)$ iff $d\big(\Lambda(A),B\big)/d\big(O,\Lambda(A)\big)\!=\!p$.  

\smallskip
(iii) The claim  follows from the definition of $B(1)$ and the fact that  $\lambda^{(n)}(p)\!=\!1$ iff $p\!=\!1/n$.

\end{proof}

\begin{corollary}
For a fixed homothetic center $O$, the  centers of mass  $A$, $\Lambda(A)$, $B(1)$, and $B$ are ordered on the line $\mathcal{L}(O,A)$ by their distances from  $O$ in the following way:
\begin{enumerate} 
\smallskip\item[(i)]$O\!<\!A\!<\!B(1)\!<\!\Lambda(A)\!<\!B$ \,\,\,if \,\,\,$\lambda\!>\!1\!+\!1/n$;
  \smallskip\item[(ii)] $O\!<\!A\!<\!B(1)\!=\!\Lambda(A)\!<\!B$ \,\,\,if \,\,\,$\lambda\!=\!1\!+\!1/n$;
  \smallskip\item [(iii)] $O\!<\!A\!<\!\Lambda(A)\!<\!B(1)\!<\!B$ \,\,\,if \,\,\,$1\!<\!\lambda\!<1\!+\!1/n$;
\smallskip\item [(iv)] $O\!<\!A\!=\!\Lambda(A)\!<\!B(1)\!\leq\! B$ \,\,\,if \,\,\,$\lambda\!=\!1;$ 
\smallskip\item [(v)] $O\!<\!\Lambda(A)\!<\!A\!<\!B(1)\!<\!B$ \,\,\,if \,\,\,$0\!<\!\lambda\!<\!1.$
\end{enumerate}

\end{corollary}

\smallskip 
We construct two representations  of  the   $\Phi^{(n)}(m)$'s  that have   clear geometric interpretations.\footnote{Note that we represent the golden ratio $\Phi\!=\!\Phi^2(1)$ geometrically using 2-dimensional sets and not intervals.} Replacing $m\!\in\!\mathbb N$ by $p\!\in\!\mathbb R^+\!$ extends the representations to the limits $\Phi^{(n)}(p)$.

First, we use  the collection of the unit $n$-balls with the centers of mass  at $(1,0,\dots,0)\!\in\!\mathbb{R}^n\!$, i.e,  the unit $k$-balls with $k\!<\!n$ can be treated as the subsets of the unit $n$-ball.  Proposition 3.1 implies that, for any $n$, there exists a dilation  about the origin in $\mathbb R^n$  such that  the unit $n$-ball and the dilated $(n\!-\!1)$-sphere enclose  the set $\mathfrak B(n)$ with the center of mass $B$    positioned at  any point in $\mathbb{R}^n\!$.

 We achieve a clear geometric interpretation of the $n$-anacci constants  if, for each $n$, we position the center of mass $B$ uniformly  at $(2,0,\dots,0)\!\in\!\mathbb{R}^n\!$.\footnote{For $n\!=\!1$, this requires $\lambda\!=\!1$, so we define $B$ as  $B(1)\!\equiv\!\lim_{1\neq\lambda\to 1}B(\lambda)\!=\!2$ and $\mathfrak B(1)\!\equiv\!\{2\}\!\in\!\mathbb{R}^1\!$.} 
Then, the dilated $n$-balls have the radii $\Phi^{(n)}(1)$ and the centers are at $(\Phi^{(n)}(1),0,\dots,0)$.  

If  $n\!>\!1$, each dilated $n$-ball  and  the unit $n$-ball  form a  connected, non-convex  set $\mathfrak B(n)$ (if  $n\!=\!1$, it is a point).   The dilated $n$-ball includes all dilated and unit $k$-balls with $k\!<\!n$, and the sets $\mathfrak B(n)$ are nested one in the other. Cf.~Fig.~5 where  $\mathfrak B(2)$ has the  shape of an eclipsed moon, and the dashed circles with diameters approaching  4 are the projections on $\mathbb R^2\!$ of the $(n\!-\!1)$-spheres enclosing the sets $\mathfrak B(n)$ with $n$=3,4,5.

The $n$-anacci sequence is represented  as  the  centers $\Phi^{(n)}(1)$ of the dilated  $n$-balls lying in the interval [1,\,2[ at the line $r${\bf e}$_1$, $r\!\in\!\mathbb R_+$, spanned by the  basic vector {\bf e}$_1$,  as well as the sequence of points $2\,\Phi^{(n)}(1)$     in the interval [2,\,4[ where this line   intersects the $(n\!-\!1)$-spheres enclosing the sets $\mathfrak B(n)$,  cf.~Fig~5, which depicts, in particular,  the constants $2\,\Phi^{(n)}(1)$ with $n\!=\!1,\dots,5$. 

This  representation  can be extended to the $(m,n)$-anacci constants $\Phi^{(n)}(m)$ in the following  way. 
We  keep the  homothetic center $O$ at the origin in $\mathbb R^n$ and move  the center of mass $B$ to $(m\!+\!1,0,\dots,0)\!\in\!\mathbb{R}^n\!,\, m\!\in\!\mathbb N$, i.e., we form the collection of   connected non-convex (if $n\!>\!1$) sets $\mathfrak B(m,n)$, each  enclosed by the unit $n$-ball and  the dilated $(n\!-\!1)$-sphere with the radius  $\Phi^{(n)}(m)$ and  the center at $\big(\Phi^{(n)}(m),0\dots,0\big)$. The inclusion of the dilated and unit $n$-balls  as well as  the sets $\mathfrak B(m,n)$ one in the other  follows the order described in the introduction. 

The $(m,n)$-anacci constants are  represented now by the centers of the dilated $n$-balls $\Phi^{(n)}(m)$ in the intervals $[m,m\!+\!1[$ and as the points $2\,\Phi^{(n)}(m)$ in the intervals  $[2\,m,2\,(m\!+\!1)[$ where the $(n\!-\!1)$-spheres enclosing the dilated $n$-balls intersect the first coordinate line $r${\bf e}$_1$. ~Fig.~5  depicts  the points $2\,\Phi^{(n)}(m)$, $m\!=\!2,3$, $n\!=\!1,\dots,5$,  in the intervals [4,\,6[ and [6,\,8[.   

 A similar representation of the $\Phi^{(n)}(m)$'s can be obtained using,  for instance, the collection of unit \linebreak $n$-cubes,  if  the homothetic center $O$ and the center of mass of a  face of  each  unit $n$-cube are located at  the origin in $\mathbb R^n\!$.  

\begin{figure}[h] 
  \centering
\vskip.05in
  \includegraphics[width=5.67in,height=2.76in,keepaspectratio]{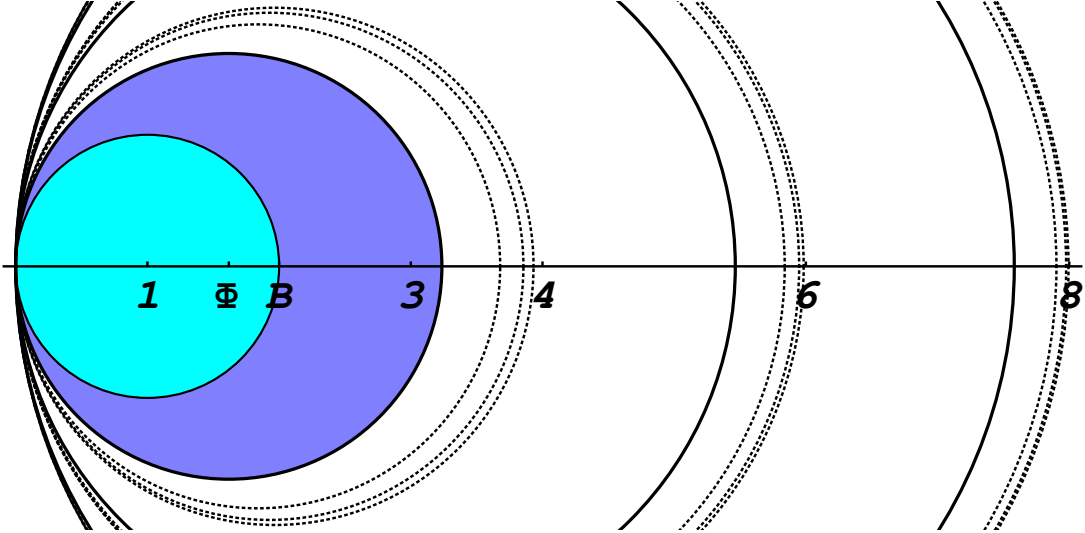}
  \caption{Geometric representation of the doubled $(m,n)$-anacci constants  $2\,\Phi^{(n)}(m)$   as the points in $\mathbb R^2\!$ where the dilated $(n\!-\!1)$-spheres  intersect the first coordinate line $r{\bf e}_1$. The shaded set $\mathfrak B(2)\!=\!\mathfrak B(1,2)$  as well as the sets  $\mathfrak B(2,2)$ and $\mathfrak B(3,2)$ are enclosed by the unit circle and the solid circles. They all  have the  shape of eclipsed moons.  The dashed circles, intersecting the line $r{\bf e}_1$, are the projections on $\mathbb R^2\!$ of the  $(n\!-\!1)$-spheres, which enclose the dilated $n$-balls corresponding to $n\!=\!3,4,5$ and $m\!=\!1,2,3$.}
\end{figure}

The fact that  the  center of mass in an  $n$-cone divides the interval between the apex  and the center of mass of the base according to the ratio  $n\!:\!1$, allows us  to construct the representation of the $(m,n)$-anacci constants $\Phi^{(n)}(m)$ that has  a clear geometric interpretation as well. 

Thus, let us consider the collection of  regular   $n$-cones with the  apexes at the origin in $\mathbb R^n\!,$   the heights equal to 1, and  the centers of mass $A(n)$ at $(\frac{n}{n+1},0,\dots,0)\!\in\!\mathbb{R}^n\!$. 
If  the radii  of the $n$-cones'  spherical bases are all  equal, the $k$-cones with $k\!<\!n$ can be treated  as the subsets of the $n$-cone.

For any $m,n\!\in\!\mathbb N$,  we  position the homothetic center of a dilation  at  $\big(\frac{ m\cdot n-1}{m\cdot (n+1)},0,\dots,0\big)\!\in\!\mathbb{R}^n\!$, and if   $m\!\cdot\! n\!\neq\!1$, we select the  dilation factor so that the center of mass $B$ of the set $\mathfrak B(m,n)$ is uniformly   at $(1,0,\dots,0)\!\in\!\mathbb{R}^n\!$, i.e., at the center of mass of the   $(n\!-\!1)$-dimensional base  of the
 $n$-cone.\footnote{For $m\!=\!n\!=\!1$,  we define $B$ as  $B(1)\!\equiv\!\lim_{1\neq\lambda\to 1}B(\lambda)\!=\!1$ and $\mathfrak B(1,1)\!\equiv\!\{1\}\!\in\!\mathbb{R}^1\!$.} If $m\cdot n\!\neq\!1$, the sets $\mathfrak B(m,n)$ are hollow, cf.~Fig.~6 that depicts  the set  $\mathfrak B(1,2)$ 

Proposition 3.1 implies that the  image  of the mass center $A(n)$ under dilation resulting in $B$ at  $(1,0,\dots,0)$ is now at the point ($\frac{\Phi^{(n)}(m)+mn-1}{m(n\!+\!1)},0,\dots,0)\!\in\!\mathbb R^n$ where  $\frac{1}{2}\!\leq\!\frac{\Phi^{(n)}(m)+mn-1}{m(n\!+\!1)}\!<\!1$. The $(m,n)$-anacci constants $\Phi^{(n)}(m)$  are  represented geometrically  by the heights of the dilated $n$-cones in a form of the closed intervals  $\big[\big(1\!-\!\Phi^{(n)}(m)\big)\frac{mn-1}{m(n+1)},\big(1\!-\!\Phi^{(n)}(m)\big)\frac{mn-1}{m(n+1)}\!+\!\Phi^{(n)}(m)\big]$ in the line $r${\bf e}$_1$ of the first coordinate. 

The ordered nesting of the intervals representing the heights of the dilated $n$-cones  holds now only for the sequences $\big(\Phi^{(n)}(m)\big)_{m=1}^{\infty}$ with a fixed  $n\!\in\!\mathbb{N}$, cf.~Appendix C. For  the sequences $\big(\Phi^{(n)}(m)\big)_{n=1}^{\infty}$ with a fixed  $m\!\in\!\mathbb{N}$, the intervals shift in the negative direction, cf.~Fig.~6, which depicts the dilated $2$-cone  with the height equal to the golden ratio $\Phi\!=\!\Phi^{(2)}(1)$.  

A similar geometric representation of the $(m,n)$-anacci constants $\Phi^{(n)}(m)$ can be constructed using, e.g., the  collection of regular  $n$-pyramids with the bases consisting of unit $(n\!-\!1)$-cubes, the apexes  at the origin in $\mathbb R^n\!$, the heights equal 1, and  the centers of mass at $(\frac{n}{n+1},0,\dots,0)\!\in\!\mathbb{R}^n\!$. 
   
\begin{figure}[h] 
  \centering
  \includegraphics[width=5.67in,height=1.57in,keepaspectratio]{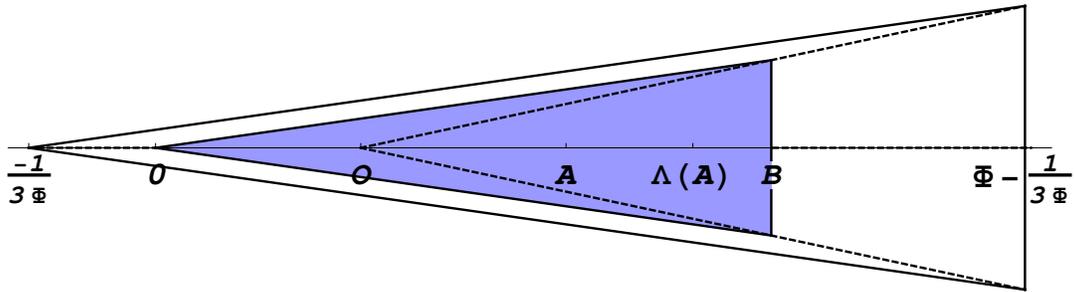}
  
\caption{Geometric representation of the golden ratio $\Phi$ by a dilation of the 2-cone.}
\end{figure}

Interestingly, Proposition 3.1 implies the theorem about the location of the centers of mass in (not-necessarily regular)  $n$-cones, $n$-pyramids, and  $n$-simplexes.
Indeed,  if dilations of a set $\mathfrak A$ about a fixed homothetic center $O$ result in  {\em convex} sets $\mathfrak B$ for  all $\lambda\!\neq\!1$, then  $B(\lambda)\!\in\!\mathfrak B$ and the center of mass  $B(1)$ lies  at the intersection of the line $\mathcal{L}(O,A)$  with  the boundary of $\mathfrak A$. 

Moreover, $B(1)$ is the center of mass of the convex part of the  boundary of $\mathfrak A$, to which the sets   $\mathfrak B$ are reduced when $\lambda\!\to\! 1$.
Thus, if we dilate  an $n$-dimensional cone, pyramid, or simplex   about the homothetic center $O$ positioned at the set's apex (vertex),  then   $B(1)$ is the center of mass of the  $(n\!-\!1)$-dimensional  face opposite to $O$, and 
 Proposition 3.1 implies  that the center of mass  $A$ of the set $\mathfrak A$ divides the distance  $d\big(O,B(1)\big)$ according to the ratio  $n\!:\!1$,  cf.~Fig.~7. 

\begin{figure}[h] 
  \centering
  \includegraphics[width=5.67in,height=1.62in,keepaspectratio]{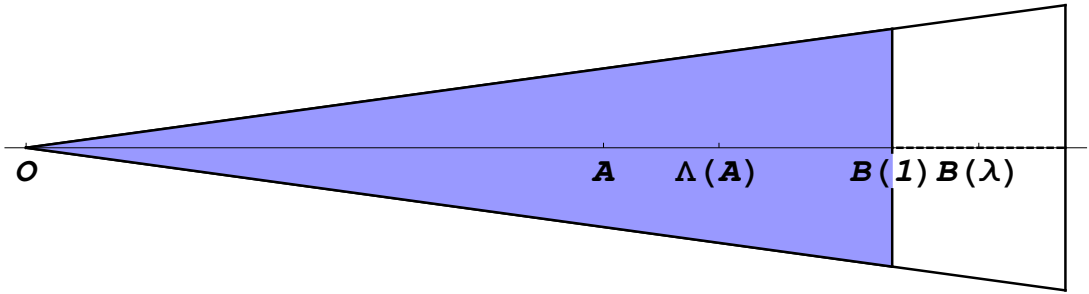}
 
\caption{The dilation of a 2-cone with the dilation factor $\lambda\!=\!1.2$ about the apex  resulting in a convex set $\mathfrak B$. If $\lambda\!\to\!1$, then $B(\lambda)\!\to\!B(1)$  and  $\Lambda(A)\!\to\!A$.}
\end{figure}

\newpage 

\noindent The proofs in the Appendices are based on the following  facts: $\Phi^{(1)}(m)\!=\!m$, and if $n\!>\!1$, 

\begin{equation}
 m\!+\!1\!-\!1\!/\!(m+1)\!<\!\Phi^{(n)}(m)\!<\!m\!+\!1. \label{37} 
\end{equation} 

\smallskip
\section{Appendix A. Proof of Corollary 2.4~(vi)}
\label{sec:AA}
  If $n\!=\!1$, the sequence $\big(\!\frac{m+1}{m}\Phi^{(1)}(m)\!\big)_{\!m=1}^{\!\infty}\!$  is strictly increasing since  $m+1\!\leq\!m+2$. If $n\!>\!1$,  \eqref{37}
\linebreak \vskip.001in
\noindent  implies that  $\frac{m+1}{m}\Phi^{(n)}(m)\!<\!\frac{(m+1)^2}{m}\!\leq\!\frac{m+2}{m+1}\big(m+2\!-\!\frac{1}{(m+2)}\big)\!<\!\frac{m\!+\!2}{m\!+\!1}\,\Phi^{(n)}(m\!+\!1)$ is true for any $m$.  

\smallskip\section{Appendix B. Proof of Corollary 2.4~(vii)}
\label{sec:AB}
 If $n\!=\!1$, $\frac{1}{m}\Phi^{(1)}(m)\!=\!1$ for any $m$. If $n\!>\!1$, the sequence $\big(\frac{1}{m}\Phi^{(n)}(m)\big)_{\!m=1}^{\!\infty}$ is strictly decreasing to  $1$
\linebreak \vskip.001in
\noindent
 since \eqref{37} implies that $\frac{m+2}{m+1}\!-\!\frac{1}{(m+2)(m+1)}\!<\!\frac{\Phi(m+1)}{m+1}\!<\!\frac{m+2}{m+1}\!\leq\!\big(m\!+\!1\!-\!\frac{1}{(m+1)}\big)\!<\!\frac{\Phi(m)}{m}\!<\!1\!+\!\frac{1}{m}$ where the middle 
\linebreak \vskip.001in
\noindent
inequality reduces to $m\!>\!1\!/\!\Phi$.   

\smallskip\section{Appendix C. 
 Proof of the nesting of the heights representing the ({\it m,n})-anacci constants}
\label {sec:AC}
For a fixed $n\!=\!1$, the nesting of the left  ends  of the intervals, which is  consistent with the introduced 
\linebreak \vskip.001in
\noindent
order, is implied by the fact that  $-\frac{m^2}{2(m+1)}\!<\!-\frac{(m-1)^2}{2m}$\, iff \,$m\!>\!1\!/\!\Phi$. For  $n\!>\!1$, formula \eqref{37} implies that 
\linebreak \vskip.001in
\noindent
$\big( 1\!-\!\Phi^{(n)}(m\!+\!1)\big)\frac{(m+1)n-1}{(m+1)(n+1)}\!<\!\frac{1-(m+1)(m+2)}{(m+2)}\cdot\frac{(m+1)n-1}{(m+1)(n+1)})\!\leq\!-m\frac{mn-1}{m(n+1)}\!<\!\big( 1\!-\!\Phi^{(n)}(m)\big)\frac{mn-1}{m(n+1)}$ is true (the 
\linebreak \vskip.001in
\noindent
middle inequality is equivalent to  
 $-(nm^2\!+\!2mn\!+\!n\!+\!1)\!\leq\!0$).  

\bigskip
The nesting of the right  ends  follows from  $\frac{m+1}{m}\Phi^{(n)}(m)\!<\!\frac{m+2}{m+1}\Phi^{(n)}(m\!+\!1)$, which  is true  for any  $n$ due to   
\linebreak \vskip.001in
\noindent Corollary 2.4 (vi), and from $\frac {mn-1}{m}\!<\!\frac{(m+1)n-1}{m+1}$, which holds for any $m$ and $n$.

\noindent MSC2010: 40A05, 40A30, 40B05

\end{document}